\documentclass[12pt, a4paper]{article}
\usepackage{}

\usepackage{amsfonts}
\usepackage{bbm}
\usepackage{mathrsfs}
\usepackage{latexsym}
\usepackage{graphicx}
\usepackage{amssymb}
\usepackage{amsmath}
\usepackage{color,xcolor}
\usepackage{mathtools}
\usepackage{cite}
\usepackage{amsthm}
\newtheorem{theorem}{Theorem}[section]
\newtheorem{lemma}[theorem]{Lemma}
\newtheorem{corollary}[theorem]{Corollary}

\newtheorem{conjecture}[theorem]{Conjecture}

\def\det{{\rm det}}

\title{{\Large \bf  The biharmonic index of connected graphs
\thanks{ Supported by the National Natural Science Foundation of China (No. 12071411).}~}}
\author{Zhen Lin\thanks{Corresponding author. E-mail addresses: lnlinzhen@163.com(Z. Lin).}\\
{\footnotesize School of Mathematics and Statistics,
Qinghai Normal University,}\\ {\footnotesize  Xining, 810008, Qinghai, P.R. China}
 }

\date{}
\begin{document}
\openup 1.0\jot
\date{}\maketitle
\begin{abstract}
Let $G$ be a simple connected graph with the vertex set $V(G)$ and $d_{B}^2(u,v)$ be the biharmonic distance between two vertices $u$ and $v$ in $G$. The biharmonic index $BH(G)$ of $G$ is defined as
$$BH(G)=\frac{1}{2}\sum\limits_{u\in V(G)}\sum\limits_{v\in V(G)}d_{B}^2(u,v)=n\sum\limits_{i=2}^{n}\frac{1}{\lambda_i^2(G)},$$
where $\lambda_i(G)$ is the $i$-th smallest eigenvalue of the Laplacian matrix of $G$ with $n$ vertices.
In this paper, we provide the mathematical relationships between the biharmonic index and some classic topological indices: the first Zagreb index, the forgotten topological index and the Kirchhoff index. In addition, the extremal value on the biharmonic index for trees and firefly graphs of fixed order are given. Finally, some graph operations on the biharmonic index are presented.

\bigskip

\noindent {\bf Mathematics Subject Classification:} 05C05; 05C09; 05C35; 05C50; 05C76

\noindent {\bf Keywords:}  Biharmonic index; Topological index; Extremal value; Graph operation

\end{abstract}
\baselineskip 20pt

\section{\large Introduction}

The Laplacian matrix of a graph $G$, denoted by $L(G)$, is given by $L(G)=D(G)-A(G)$, where $D(G)$ is the diagonal matrix of its vertex degrees and $A(G)$ is the adjacency matrix. The Laplacian characteristic polynomial of $G$, is equal to $\det(xI_n-L(G))$, denoted by $\phi(L(G))$. We denote
$\lambda_i=\lambda_i(G)$ the $i$-th smallest eigenvalue of $L(G)$. In particular, $\lambda_2(G)$ and $\lambda_n(G)$ are called the algebraic connectivity \cite{F} and the Laplacian spectral radius of $G$, respectively. The Laplacian spectral ratio of a connected graph $G$ with $n$ vertices is defined as
$r_L(G)=\frac{\lambda_n}{\lambda_2}$. Barahona et al. \cite{BP} showed that a graph $G$ exhibits better synchronizability if the ratio $r_L(G)$ is as small as possible.

The topological indices have fundamental applications in chemical disciplines \cite{D, DB, PLC}, computational linguistics \cite{MWL}, computational biology \cite{MKGD} and etc. Let $d(u, v)$ be the distance between vertices $u$ and $v$ of $G$. The Wiener index $W(G)$ of a connected graph $G$, introduced by Wiener \cite{W} in 1947, is defined as $W(G)=\sum_{u, v\in V(G)}d(u, v)$, which is used to predict the boiling points of paraffins by their molecular structure.
The Wiener index found numerous applications in pure mathematics and other sciences \cite{GKM, KST}. In 1972, Gutman and Trinajsti\'{c} \cite{GT} proposed the first Zagreb index $M_1(G)$ of a graph $G$, and defined it as the sum of the squares of vertex degrees of $G$. There is a wealth of literature relating to the first Zagreb index, the reader is referred to the survey \cite{BDFG, FT} and the references therein. Recently, Furtula and Gutman \cite{FG} defined the forgotten topological index of a graph $G$ as the sum of the cubes of vertex degrees of $G$, denoted by $F(G)$. In particular, the forgotten topological index of several important chemical structures which have high frequency in drug structures is obtained \cite{GSIJF}. The Kirchhoff index of a graph $G$ is defined as the sum of resistance distances \cite{KR} between all pairs of vertices of $G$, denoted by $Kf(G)$. Gutman and Mohar \cite{GM1} gave an important calculation formula on Kirchhoff index, that is $Kf(G)=\sum_{i=2}^{n}\frac{1}{\lambda_i}$. The Kirchhoff index is often used to measure how well connected a network is \cite{GBS, KR}.

In 2010, Lipman, Rustamov and Funkhouser \cite{LRF} proposed the biharmonic distance to measure the distances between pairs of points on a 3D surface, which is a fundamental problem in computer graphics and geometric processing. Moreover, the biharmonic distance has some advantages over resistance distance and geodesic distance in computer graphics, geometric processing, shape analysis and etc.
Inspired by Wiener index, Wei, Li and Yang \cite{WLY} proposed the concept of biharmonic index of a graph $G$ as follows:
$$BH(G)=\frac{1}{2}\sum\limits_{u\in V(G)}\sum\limits_{v\in V(G)}d_{B}^2(u,v)=n\sum\limits_{i=2}^{n}\frac{1}{\lambda_i^2(G)}.$$
Meanwhile, they obtained a relationship between biharmonic index and Kirchhoff index and determined the unique graph having the minimum biharmonic index among the connected graphs with $n$ vertices.

In this paper, we continue the work on the biharmonic index of connected graphs. Firstly, we establish the mathematical relationships between the biharmonic index and some classic topological indices: the first Zagreb index, the forgotten topological index and the Kirchhoff index. Secondly, we study the extremal value on the biharmonic index for trees and firefly graphs of fixed order, around Problem 6.3 in \cite{WLY}. Meanwhile, we show that the star is the unique graph with maximum biharmonic index among all graphs on diameter two. Finally, some graph operations on the biharmonic index are presented.

\section{\large  Preliminaries}

Let $K_{1, \, n-1}$, $P_n$ and $K_n$ denote the star, the path and the complete graph with $n$ vertices, respectively. Let $\tau(G)$ be the number of spanning trees of a connected graph.
The double star $S(a,\, b)$ is the tree obtained from $K_2$ by attaching $a$ pendant edges to a vertex and $b$ pendant edges to the other. A firefly graph $F_{s,\,t,\,n-2s-2t-1}$ ($s\geq0, \, t\geq 0, \, n-2s-2t-1\geq 0$) is a graph of order $n$ that consists of $s$ triangles, $t$ pendent paths of length $2$ and $n-2s-2t-1$ pendent edges, sharing a common vertex. For $v\in V(G)$, let $L_v(G)$ be the principal submatrix of $L(G)$ formed by deleting the row and column corresponding to vertex $v$.

\begin{lemma}{\bf (\cite{PS})}\label{le2,1} 
Let $\overline{a}=(a_1, \ldots, a_n)$ and $\overline{b}=(b_1, \ldots, b_n)$ be two positive $n$-tuples. Then
$$\frac{\left(\sum\limits_{i=1}^{n}a_i^2\right)\left(\sum\limits_{i=1}^{n}b_i^2\right)}{\left(\sum\limits_{i=1}^{n}a_ib_i\right)^2}\leq \frac{(a+A)^2}{4aA},$$
where $a=\min\{\frac{a_i}{b_i}\}$ and $A=\max\{\frac{a_i}{b_i}\}$ for $1\leq i \leq n$.
\end{lemma}

\begin{lemma}{\bf (\cite{PS})}\label{le2,2} 
Let $\overline{a}=(a_1, \ldots, a_n)$ and $\overline{b}=(b_1, \ldots, b_n)$ be two positive $n$-tuples. Then
$$\left(\sum\limits_{i=1}^{n}a_i^2\right)\left(\sum\limits_{i=1}^{n}b_i^2\right)-\left(\sum\limits_{i=1}^{n}a_ib_i\right)^2\leq \frac{(A-a)^2}{4aA}\left(\sum\limits_{i=1}^{n}a_ib_i\right)^2,$$
where $a=\min\{\frac{a_i}{b_i}\}$ and $A=\max\{\frac{a_i}{b_i}\}$ for $1\leq i \leq n$.
\end{lemma}

\begin{lemma}{\bf (\cite{R})}\label{le2,3} 
If $a_i>0$, $b_i>0$, $p>0$, $i=1, 2, \ldots, n$, then the following inequality holds:
$$\sum\limits_{i=1}^{n}\frac{a_i^{p+1}}{b_i^p}\geq \frac{\left(\sum\limits_{i=1}^{n}a_i\right)^{p+1}}{\left(\sum\limits_{i=1}^{n}b_i\right)^{p}}$$
with equality if and only if $\frac{a_1}{b_1}=\frac{a_2}{b_2}=\cdots=\frac{a_n}{b_n}$.
\end{lemma}

\begin{lemma}{\bf (\cite{O})}\label{le2,4} 
Let $n\geq 1$ be an integer and $a_1\geq a_2\geq \cdots \geq a_n$ be some non-negative real numbers. Then
$$(a_1+a_n)(a_1+a_2+\cdots+a_n)\geq a_1^2+a_2^2+\cdots+a_n^2+na_1a_n$$
Moreover, the equality holds if and only if for some $r\in \{1,2,\ldots,n\}$, $a_1=\cdots=a_r$ and
$a_{r+1}=\cdots=a_n$.
\end{lemma}

\begin{lemma}{\bf (\cite{LGSBB})}\label{le2,5} 
Let $a_1, \ldots, a_n\geq 0 $. Then
$$n\left[\frac{1}{n}\sum\limits_{i=1}^{n}a_i-\left(\prod\limits_{i=1}^{n}a_i\right)^{\frac{1}{n}}\right]\leq \Phi \leq n(n-1)\left[\frac{\sum\limits_{i=1}^{n}a_i}{n}-\left(\prod\limits_{i=1}^{n}a_i\right)^{\frac{1}{n}}\right],$$
where $\Phi=n\sum_{i=1}^{n}a_i-\left(\sum_{i=1}^{n}\sqrt{a_i}\right)^2$.
\end{lemma}

\begin{lemma}{\bf (\cite{LMR})}\label{le2,6} 
Let $a_1, a_2, \ldots, a_n$ and $b_1, b_2, \ldots, b_n$ be real numbers such
that $a\leq a_i\leq A$ and $b\leq b_i\leq B$ for $i=1, 2, \ldots, n$. Then there holds
$$\left\lvert \frac{1}{n}\sum\limits_{i=1}^{n}a_ib_i-\left(\frac{1}{n}\sum\limits_{i=1}^{n}a_i\right)\left(\frac{1}{n}\sum\limits_{i=1}^{n}b_i\right)\right\rvert\leq \frac{1}{n}\left\lfloor\frac{n}{2}\right\rfloor\left(1-\frac{1}{n}\left\lfloor\frac{n}{2}\right\rfloor\right)(A-a)(B-b),$$
where $\lfloor x\rfloor$ denotes the integer part of $x$.
\end{lemma}

\begin{lemma}{\bf (\cite{GMS})}\label{le2,7} 
If $T$ is a tree with diameter $d(T)$, then $\lambda_2(T)\leq 2\left(1-\cos\left(\frac{\pi}{d+1}\right)\right)$.
\end{lemma}

\begin{lemma}{\bf (\cite{JOT})}\label{le2,8} 
The number of Laplacian eigenvalues less than the average degree $2-\frac{2}{n}$ of a tree with $n$ vertices is at least $\lceil \frac{n}{2}\rceil$.
\end{lemma}

\begin{lemma}{\bf (\cite{D1})}\label{le2,9} 
Let $G$ be a connected graph of diameter $2$. Then $\lambda_2(G)\geq 1$.
\end{lemma}

\begin{lemma}{\bf (\cite{G})}\label{le2,10} 
Let $uv$ be a cut edge of a graph $G$. Let $G-uv=G_1+G_2$, where $G_1$ and $G_2$ are the
components of $G-uv$, $G_1+G_2$ is the sum of $G_1$ and $G_2$, $u\in V(G_1)$ and $v\in V(G_2)$. Then
$$\phi(L(G))=\phi(L(G_1))\phi(L(G_2))-\phi(L(G_1))\phi(L_v(G_2))-\phi(L_u(G_1))\phi(L(G_2)).$$
\end{lemma}

\section{\large  The biharmonic index, the first Zagreb index and the forgotten topological index}

\begin{theorem}\label{th3,1} 
Let $G$ be a connected graph with $n$ vertices and $m$ edges. Then
$$BH(G)\leq \frac{n(n-1)^2}{4(2m+M_1(G))}\left(r_L(G)+\frac{1}{r_L(G)}\right)^{2}.$$
\end{theorem}

\begin{proof}  In this proof we use Lemma \ref{le2,1} with $a_i=\lambda_i$ and $b_i=\frac{1}{\lambda_i}$ for $2\leq i \leq n$. Then $a=\lambda_2^2$ and $A=\lambda_n^2$. Thus
$$\frac{\left(\sum\limits_{i=2}^{n}\lambda_i^2\right)\left(\sum\limits_{i=2}^{n}\frac{1}{\lambda_i^2}\right)}{(n-1)^2}\leq \frac{(\lambda_2^2+\lambda_n^2)^2}{4\lambda_2^2\lambda_n^2},$$
Since $\sum_{i=2}^{n}\lambda_i^2=2m+M_1(G)$, we have
$$\frac{(2m+M_1(G))BH(G)}{n(n-1)^2}\leq \frac{(\lambda_2^2+\lambda_n^2)^2}{4\lambda_2^2\lambda_n^2},$$
that is,
$$BH(G)\leq \frac{n(n-1)^2}{4(2m+M_1(G))}\left(\frac{\lambda_2}{\lambda_n}+\frac{\lambda_n}{\lambda_2}\right)^{2}.$$
This completes the proof. $\Box$
\end{proof}

\begin{theorem}\label{th3,2} 
Let $G$ be a connected graph with $n$ vertices and $m$ edges. Then
$$BH(G)\leq \frac{n(n-1)^2}{4(2m+M_1(G))}\left(4+\left(r_L(G)-\frac{1}{r_L(G)}\right)^{2}\right).$$
\end{theorem}

\begin{proof}  In this proof we use Lemma \ref{le2,2} with $a_i=\lambda_i$ and $b_i=\frac{1}{\lambda_i}$ for $2\leq i \leq n$. Then $a=\lambda_2^2$ and $A=\lambda_n^2$. Thus
$$\left(\sum\limits_{i=2}^{n}\lambda_i^2\right)\left(\sum\limits_{i=2}^{n}\frac{1}{\lambda_i^2}\right)-(n-1)^2\leq \frac{(\lambda_n^2-\lambda_2^2)^2}{4\lambda_2^2\lambda_n^2}(n-1)^2.$$
Since $\sum_{i=2}^{n}\lambda_i^2=2m+M_1(G)$, we have
$$\frac{2m+M_1(G)}{n}BH(G)-(n-1)^2\leq \frac{(\lambda_n^2-\lambda_2^2)^2}{4\lambda_2^2\lambda_n^2}(n-1)^2,$$
that is,
$$BH(G)\leq \frac{n(n-1)^2}{4(2m+M_1(G))}\left(4+\left(\frac{\lambda_n}{\lambda_2}-\frac{\lambda_2}{\lambda_n}\right)^{2}\right).$$
This completes the proof. $\Box$
\end{proof}

\begin{theorem}\label{th3,3} 
Let $p$ be a positive real number and $G$ be a connected graph with $n$ vertices and $m$ edges. Then
$$BH(G)\geq n\left(\frac{(2m)^{p+1}}{\sum\limits_{i=2}^{n}\lambda_i^{3p+1}}\right)^{\frac{1}{p}}$$
with equality if and only if $G\cong K_n$.
\end{theorem}

\begin{proof} In this proof we use Lemma \ref{le2,3} with $a_i=\lambda_i$ and $b_i=\frac{1}{\lambda_i^2}$ for $2\leq i \leq n$ . Then we have
$$\sum\limits_{i=2}^{n}\lambda_i^{3p+1}\geq \frac{\left(\sum\limits_{i=2}^{n}\lambda_i\right)^{p+1}}{\left(\sum\limits_{i=2}^{n}\frac{1}{\lambda_i^2}\right)^{p}}.$$
Since $\sum_{i=2}^{n}\lambda_i=2m$, we have
$$BH(G)\geq n\left(\frac{(2m)^{p+1}}{\sum\limits_{i=2}^{n}\lambda_i^{3p+1}}\right)^{\frac{1}{p}}$$
with equality if and only if $\lambda_2^3=\cdots=\lambda_n^3$, that is $G\cong K_n$.
This completes the proof. $\Box$
\end{proof}

\begin{corollary}\label{cor3,1} 
Let $G$ be a connected graph with $n$ vertices and $m$ edges. Then
$$BH(G)\geq \frac{16nm^4}{[2m+M_1(G)]^3}$$
with equality if and only if $G\cong K_n$.
\end{corollary}

\begin{proof} Let $p=\frac{1}{3}$. Since $\sum_{i=2}^{n}\lambda_i^{2}=2m+M_1(G)$, by Theorem \ref{th3,3}, we have
$$BH(G)\geq n\left(\frac{(2m)^{4/3}}{\sum\limits_{i=2}^{n}\lambda_i^{2}}\right)^{3}=\frac{16nm^4}{[2m+M_1(G)]^3}$$
with equality if and only if $G\cong K_n$. This completes the proof. $\Box$
\end{proof}

\begin{corollary}\label{cor3,2} 
Let $G$ be a connected graph with $n$ vertices, $m$ edges and $t(G)$ triangles. Then
$$BH(G)\geq \sqrt{\frac{32n^2m^5}{[3M_1(G)+F(G)+6t(G)]^3}}$$
with equality if and only if $G\cong K_n$.
\end{corollary}

\begin{proof} Let $p=\frac{2}{3}$. Since $\sum_{i=2}^{n}\lambda_i^{3}=3M_1(G)+F(G)+6t(G)$, by Theorem \ref{th3,3}, we have
$$BH(G)\geq n\left(\frac{(2m)^{5/3}}{\sum\limits_{i=2}^{n}\lambda_i^{3}}\right)^{3/2}=\sqrt{\frac{32n^2m^5}{[3M_1(G)+F(G)+6t(G)]^3}}$$
with equality if and only if $G\cong K_n$. This completes the proof. $\Box$
\end{proof}

\section{\large  The biharmonic index and Kirchhoff index }

\begin{theorem}\label{th4,1} 
Let $G$ be a connected graph with $n$ vertices. Then
$$BH(G)\leq \left(\frac{1}{\lambda_2}+\frac{1}{\lambda_n}\right)Kf(G)-n(n-1)\frac{1}{\lambda_2\lambda_n}$$
with equality if and only if for some $r\in \{2,\ldots,n\}$, $\lambda_2=\cdots=\lambda_r$ and
$\lambda_{r+1}=\cdots=\lambda_n$.
\end{theorem}

\begin{proof} By Lemma \ref{le2,4}, we have
$$\left(\frac{1}{\lambda_2}+\frac{1}{\lambda_n}\right)\left(\frac{1}{\lambda_2}+\cdots+\frac{1}{\lambda_n}\right)\geq \frac{1}{\lambda_2^2}+\cdots+\frac{1}{\lambda_n^2}+(n-1)\frac{1}{\lambda_2\lambda_n},$$
that is,
$$n\left(\frac{1}{\lambda_2}+\frac{1}{\lambda_n}\right)\left(\frac{1}{\lambda_2}+\cdots+\frac{1}{\lambda_n}\right)\geq n\left(\frac{1}{\lambda_2^2}+\cdots+\frac{1}{\lambda_n^2}\right)+n(n-1)\frac{1}{\lambda_2\lambda_n},$$
that is,
$$\left(\frac{1}{\lambda_2}+\frac{1}{\lambda_n}\right)Kf(G)\geq nBH(G)+n(n-1)\frac{1}{\lambda_2\lambda_n},$$
that is,
$$BH(G)\leq \left(\frac{1}{\lambda_2}+\frac{1}{\lambda_n}\right)Kf(G)-n(n-1)\frac{1}{\lambda_2\lambda_n}$$
with equality if and only if for some $r\in \{2,\ldots,n\}$, $\lambda_2=\cdots=\lambda_r$ and
$\lambda_{r+1}=\cdots=\lambda_n$. This completes the proof. $\Box$
\end{proof}

\begin{theorem}\label{th4,2} 
Let $G$ be a connected graph with $n\geq 3$ vertices. Then
$$\frac{Kf^2(G)}{n(n-2)}-\frac{n(n-1)}{n-2}\left(\frac{1}{n\tau(G)}\right)^{\frac{2}{n-1}}\leq BH(G)\leq \frac{Kf^2(G)}{n}-n(n-1)(n-2)\left(\frac{1}{n\tau(G)}\right)^{\frac{2}{n-1}}.$$
\end{theorem}

\begin{proof} In this proof we use Lemma \ref{le2,5} with $a_i=\frac{1}{\lambda_i^2}$ for $2\leq i \leq n$ . Then we have
$$(n-1)\left[\frac{1}{n-1}\sum\limits_{i=2}^{n}\frac{1}{\lambda_i^2}-\left(\prod\limits_{i=2}^{n}\frac{1}{\lambda_i^2}\right)^{\frac{1}{n-1}}\right]\leq \Phi \leq (n-1)(n-2)\left[\frac{\sum\limits_{i=2}^{n}\frac{1}{\lambda_i^2}}{n-1}-\left(\prod\limits_{i=2}^{n}\frac{1}{\lambda_i^2}\right)^{\frac{1}{n-1}}\right],$$
where $\Phi=(n-1)\sum_{i=2}^{n}\frac{1}{\lambda_i^2}-\left(\sum_{i=2}^{n}\sqrt{\frac{1}{\lambda_i^2}}\right)^2=\frac{n-1}{n}BH(G)-\frac{1}{n^2}Kf^2(G)$. Since $\prod_{i=2}^{n}\lambda_i=n\tau(G)$, we have
$$\frac{1}{n}BH(G)-(n-1)\left(\frac{1}{n\tau(G)}\right)^{\frac{2}{n-1}}\leq \Phi \leq \frac{n-2}{n}BH(G)-(n-1)(n-2)\left(\frac{1}{n\tau(G)}\right)^{\frac{2}{n-1}},$$
where $\Phi=\frac{n-1}{n}BH(G)-\frac{1}{n^2}Kf^2(G)$. Thus we have
$$\frac{Kf^2(G)}{n(n-2)}-\frac{n(n-1)}{n-2}\left(\frac{1}{n\tau(G)}\right)^{\frac{2}{n-1}}\leq BH(G)\leq \frac{Kf^2(G)}{n}-n(n-1)(n-2)\left(\frac{1}{n\tau(G)}\right)^{\frac{2}{n-1}}.$$
This completes the proof. $\Box$
\end{proof}

\begin{theorem}\label{th4,3} 
Let $G$ be a connected graph with $n$ vertices. Then
$$\left\lvert n(n-1)BH(G)-Kf^2(G)\right\rvert\leq \frac{n^2(n-1)^2}{4}\left(1-\frac{1+(-1)^{n+1}}{2n^2}\right)\left(\frac{1}{\lambda_2}
-\frac{1}{\lambda_n}\right)^2.$$
\end{theorem}

\begin{proof} In this proof we use Lemma \ref{le2,6} with $a_i=b_i=\frac{1}{\lambda_i}$ for $2\leq i \leq n$ . Then we have
$$\left\lvert \frac{1}{n-1}\sum\limits_{i=2}^{n}\frac{1}{\lambda_i^2}-\frac{1}{(n-1)^2}\sum\limits_{i=2}^{n}\frac{1}{\lambda_i}\right\rvert\leq \frac{1}{n}\left\lfloor\frac{n}{2}\right\rfloor\left(1-\frac{1}{n}\left\lfloor\frac{n}{2}\right\rfloor\right)\left(\frac{1}{\lambda_2}
-\frac{1}{\lambda_n}\right)^2,$$
that is,
$$\left\lvert n(n-1)BH(G)-Kf^2(G)\right\rvert\leq n(n-1)^2\left\lfloor\frac{n}{2}\right\rfloor\left(1-\frac{1}{n}\left\lfloor\frac{n}{2}\right\rfloor\right)\left(\frac{1}{\lambda_2}
-\frac{1}{\lambda_n}\right)^2.$$
Note that $\left\lfloor\frac{n}{2}\right\rfloor\left(1-\frac{1}{n}\left\lfloor\frac{n}{2}\right\rfloor\right)
=\frac{n}{4}\left(1-\frac{1+(-1)^{n+1}}{2n^2}\right)$. We have
$$\left\lvert n(n-1)BH(G)-Kf^2(G)\right\rvert\leq \frac{n^2(n-1)^2}{4}\left(1-\frac{1+(-1)^{n+1}}{2n^2}\right)\left(\frac{1}{\lambda_2}
-\frac{1}{\lambda_n}\right)^2.$$
This completes the proof. $\Box$
\end{proof}

\section{\large  The biharmonic index of trees and firefly graphs  }

\begin{theorem}\label{th5,1} 
Let $S(a, b)$ be a double star tree on $n$ vertices and $a+b=n-2$. Then
$$n^2+3n+\frac{4}{n}-16\leq BH(S(a,b))\leq n^2-2n+4\left\lceil\frac{n-2}{2}\right\rceil\left\lfloor\frac{n-2}{2}\right\rfloor+\frac{(\lceil\frac{n-2}{2}\rceil\lfloor\frac{n-2}{2}\rfloor+1)^2}{n},$$
the left (right) equality holds if and only if $S(1,n-3)$ $(S(\lceil\frac{n-2}{2}\rceil, \lfloor\frac{n-2}{2}\rfloor)$.
\end{theorem}

\begin{proof} By direct calculation, we have
$$\phi(L(S(a,b)))=x(x-1)^{n-4}[x^3-(n+2)x^2+(2n+ab+1)x-n].$$
Let $x_1$, $x_2$ and $x_3$ be the roots of the following polynomial
$$f(x):= x^3-(n+2)x^2+(2n+ab+1)x-n.$$
By the Vieta Theorem, we have
$$
\begin{cases}
x_1+x_2+x_3=n+2, \\
\frac{1}{x_1}+\frac{1}{x_2}+\frac{1}{x_3}=\frac{2n+ab+1}{n},\\
x_1x_2x_3=n.
\end{cases} $$
Thus
\begin{eqnarray*}
\frac{1}{x_1^2}+\frac{1}{x_2^2}+\frac{1}{x_3^2} & = & \left(\frac{1}{x_1}+\frac{1}{x_2}+\frac{1}{x_3}\right)^2
-2\left(\frac{1}{x_1x_2}+\frac{1}{x_2x_3}+\frac{1}{x_1x_3}\right)\\
& = & \left(\frac{1}{x_1}+\frac{1}{x_2}+\frac{1}{x_3}\right)^2
-\frac{2}{n}(x_1+x_2+x_3)\\
& = & \left(\frac{2n+ab+1}{n}\right)^2
-\frac{2}{n}(n+2)\\
& = & \left(\frac{2n+ab+1}{n}\right)^2-\frac{4}{n}-2.
\end{eqnarray*}
Further, we have
$$BH(S(a,b))=n\sum\limits_{i=2}^{n}\frac{1}{\lambda_i^2}=n(n-4)+2n+4ab+\frac{(ab+1)^2}{n}=n^2-2n+4ab+\frac{(ab+1)^2}{n}.$$
Since $n-3\leq ab \leq \lceil\frac{n-2}{2}\rceil\lfloor\frac{n-2}{2}\rfloor$, we have
$$n^2+3n+\frac{4}{n}-16\leq BH(S(a,b))\leq n^2-2n+4\left\lceil\frac{n-2}{2}\right\rceil\left\lfloor\frac{n-2}{2}\right\rfloor+\frac{(\lceil\frac{n-2}{2}\rceil\lfloor\frac{n-2}{2}\rfloor+1)^2}{n},$$
the left (right) equality holds if and only if $S(1,n-3)$ $(S(\lceil\frac{n-2}{2}\rceil, \lfloor\frac{n-2}{2}\rfloor)$. This completes the proof. $\Box$
\end{proof}

\begin{theorem}\label{th5,2} 
Let $T_n$ be a tree on $n\geq 8$ vertices. If the diameter $d(T_n)\geq \pi \sqrt[4]{\frac{7n}{8}}-1$, then
$$BH(T_n)> BH(K_{1,\, n-1}).$$
\end{theorem}

\begin{proof} Since $1-\cos x<\frac{x^2}{2}$, by Lemma \ref{le2,7}, we have
$$\lambda_2(T_n)\leq 2\left(1-\cos\left(\frac{\pi}{d(T_n)+1}\right)\right)<\left(\frac{\pi}{d(T_n)+1}\right)^2$$
By Lemma \ref{le2,8}, we have
\begin{eqnarray*}
BH(T_n) & = & n\left(\frac{1}{\lambda_2^2}+\cdots+\frac{1}{\lambda_n^2}\right)\\
& > & n\left(\frac{(d(T_n)+1)^4}{\pi^4}+\left(\left\lceil \frac{n}{2}\right\rceil-2\right)\frac{1}{\left(2-\frac{2}{n}\right)^2}+\left\lfloor\frac{2}{n}\right\rfloor\frac{1}{n^2}\right)\\
& > & n\left(\frac{(d(T_n)+1)^4}{\pi^4}+\left(\frac{n}{2}-2\right)\frac{1}{\left(2-\frac{2}{n}\right)^2}+\left(\frac{2}{n}-1\right)\frac{1}{n^2}\right)\\
& = & n\left(\frac{(d(T_n)+1)^4}{\pi^4}+\frac{n^2(n-4)}{8(n-1)^2}+\left(\frac{2}{n}-1\right)\frac{1}{n^2}\right)\\
& \geq & n\left(\frac{7n}{8}+\frac{n^2(n-4)}{8(n-1)^2}+\left(\frac{2}{n}-1\right)\frac{1}{n^2}\right)\\
& > & n(n-1)\\
& > & n\left(n-2+\frac{1}{n^2}\right)\\
& = & BH(K_{1,\,n-1})
\end{eqnarray*}
for $n\geq 8$. This completes the proof. $\Box$
\end{proof}

The following conjecture is concretization of Problem 6.3 in \cite{WLY}.

\begin{conjecture}
Let $T_n$ be a tree on $n\geq 5$ vertices. Then
$$BH(K_{1,\,n-1})\leq BH(T_n)\leq BH(P_n),$$
the left (right) equality holds if and only if  $T_n=K_{1,\,n-1}$ ($T_n=P_n$).
\end{conjecture}

\begin{theorem}\label{th5,3} 
Let $G$ be a connected graph with $n$ vertices and diameter $d(G)=2$. Then
$$BH(G)\leq BH(K_{1,\,n-1})$$
with equality if and only if $G=K_{1,\,n-1}$.
\end{theorem}

\begin{proof} It is well known that $\lambda_n\geq \Delta+1$ and $\lambda_{n-1}\geq \Delta_2$ (see \cite{LP, GM}), where $\Delta$ and $\Delta_2$ are the maximum degree and the second largest degree of $G$, respectively. If $2\leq\Delta_2\leq \Delta$, by Lemma \ref{le2,9}, we have
\begin{eqnarray*}
BH(G) & = & n\left(\frac{1}{\lambda_2^2}+\cdots+\frac{1}{\lambda_n^2}\right)\\
& \leq & n \left(n-3+\frac{1}{\Delta_2^2}+\frac{1}{(\Delta+1)^2}\right)\\
& < &  n \left(n-3+\frac{1}{2^2}+\frac{1}{(2+1)^2}\right)\\
& < & n \left(n-2+\frac{1}{n^2}\right)\\
& = & BH(K_{1,\,n-1}).
\end{eqnarray*}
Thus $\Delta_2=1$, that is, $G=K_{1,\,n-1}$, then $BH(G)=BH(K_{1,\,n-1})$.

Combining the above arguments, we have $BH(G)\leq BH(K_{1,\,n-1})$
with equality if and only if $G=K_{1,\,n-1}$. This completes the proof. $\Box$
\end{proof}

\begin{theorem}\label{th5,4} 
Let $F_{s,\,t,\,n-2s-2t-1}$ ($s\geq0, \, t\geq 0, \, n-2s-2t-1\geq 0$) be a firefly graph with $n\geq 7$ vertices.

(1) If $s=t=0$, then $BH(F_{0,\, 0,\, n-1})=n^2-2n+\frac{1}{n}$.

(2) If $s=0$ and $t=1$, then $BH(F_{0,\,1,\,n-3})=n^2+3n-16+\frac{4}{n}$.

(3) If $s=0$, $t\geq 2$ and $n$ is odd, then
$$n^2+8n+\frac{25}{n}-32\leq BH(F_{0,\,t,\,n-2t-1})\leq \frac{7n^2 }{2}-\frac{41n}{4}+\frac{25}{4n}+\frac{1}{2},$$
the left (right) equality holds if and only if  $F_{0,\,t,\,n-2t-1}=F_{0,\,2,\,n-5}$ ($F_{0,\,t,\,n-2t-1}=F_{0,\, \frac{n-1}{2},\,0}$). \\
If $s=0$, $t\geq 2$ and $n$ is even, then
$$n^2+8n+\frac{25}{n}-32\leq BH(F_{0,\,t,\,n-2t-1})\leq \frac{7n^2}{2}-\frac{51n}{4}+\frac{16}{n}+4,$$
the left (right) equality holds if and only if  $F_{0,\,t,\,n-2t-1}=F_{0,\,2,\,n-5}$ ($F_{0,\,t,\,n-2t-1}=F_{0,\, \frac{n-2}{2},\,1}$).

(4) If $s\geq 1$, $t=0$ and $n$ is odd, then
$$\frac{5n^2 }{9}-\frac{14n}{9}+\frac{1}{n}\leq BH(F_{0,\,t,\,n-2t-1})\leq n^2-\frac{26}{9}n+\frac{1}{n},$$
the left (right) equality holds if and only if $F_{s,\,0,\,n-2s-1}=F_{1,\,0,\,n-3}$ ($F_{s,\,0,\,n-2s-1}=F_{\frac{n-1}{2},\,0,\,0}$).\\
If $s\geq 1$, $t=0$ and $n$ is even, then
$$\frac{5n^2}{9}-\frac{10n}{9}+\frac{1}{n}\leq BH(F_{0,\,t,\,n-2t-1})\leq n^2-\frac{26}{9}n+\frac{1}{n},$$
the left (right) equality holds if and only if $F_{s,\,0,\,n-2s-1}=F_{1,\,0,\,n-3}$ ($F_{s,\,0,\,n-2s-1}=F_{\frac{n-2}{2},\,0,\,1}$).

(5) If $s\geq 1$, $t\geq 1$ and $n$ is odd, then
$$\frac{5n^2}{9}+\frac{13n}{3}+\frac{4}{n}-16\leq BH(F_{s,\,t,\,n-2s-2t-1})\leq \frac{7n^2}{2}-\frac{581n}{36}+\frac{121}{4n}+\frac{15}{2},$$
the left (right) equality holds if and only if $F_{s,\,t,\,n-2s-2t-1}=F_{\frac{n-3}{2},\,1,\,0}$ ($F_{s,\,t,\,n-2s-2t-1}=F_{1,\,\frac{n-3}{2},\,0}$).\\
If $s\geq 1$, $t\geq 1$ and $n$ is even, then
$$\frac{5n^2}{9}+\frac{43n}{9}+\frac{4}{n}-16\leq BH(F_{s,\,t,\,n-2s-2t-1})\leq \frac{7n^2}{2}-\frac{671n}{36}+\frac{49}{n}+11,$$
the left (right) equality holds if and only if $F_{s,\,t,\,n-2s-2t-1}=F_{\frac{n-4}{2},\,1,\,1}$ ($F_{s,\,t,\,n-2s-2t-1}=F_{1,\,\frac{n-4}{2},\,1}$).
\end{theorem}

\begin{proof} (1) If $s=t=0$, then $F_{0,\, 0,\, n-1}\cong K_{1,\,n-1}$. Thus $BH(F_{0,\, 0,\, n-1})=n^2-2n+\frac{1}{n}$.

(2) If $s=0$ and $t=1$, by Lemma \ref{le2,10}, we have
\begin{eqnarray*}
\phi(L(F_{0,\,1,\,n-3})) & = & \phi(L(K_{1,\,n-3}))\phi(L(P_2))-(x-1)^{n-3}\phi(L(P_2))-(x-1)\phi(L(K_{1,\,n-3}))\\
& = & x^2(x-2)(x-n+2)(x-1)^{n-4}-x(x-2)(x-1)^{n-3}\\
& & -x(x-n+2)(x-1)^{n-3}\\
& = & x(x-1)^{n-4}[x^3-(n+2)x^2+(3n-2)x-n].
\end{eqnarray*}
By a similar reasoning as the proof of Theorem \ref{th5,1}, we have
$$BH(F_{0,\,1,\,n-3})=n^2+3n-16+\frac{4}{n}.$$

(3) If $s=0$ and $t\geq 2$, then we have
$$\phi(L(F_{0,\,t,\,n-2t-1}))=x(x-1)^{n-2t-2}(x^2-3x+1)^{t-1}[x^3-(n-t+3)x^2+(3n-3t+1)x-n].$$
By a similar reasoning as the proof of Theorem \ref{th5,1}, we have
\begin{eqnarray*}
BH(F_{0,\,t,\,n-2t-1}) & = & n^2+5tn-11n+2t+\frac{(3n-3t+1)^2}{n}-6\\
& = & n^2-2n+\frac{9t^2+(5n^2-16n-6)t+1}{n}.
\end{eqnarray*}
If $2\leq t\leq \frac{n-1}{2}$ for odd $n$, we have
$$n^2+8n+\frac{25}{n}-32\leq BH(F_{0,\,t,\,n-2t-1})\leq \frac{7n^2 }{2}-\frac{41n}{4}+\frac{25}{4n}+\frac{1}{2},$$
the left (right) equality holds if and only if  $F_{0,\,t,\,n-2t-1}=F_{0,\,2,\,n-5}$ ($F_{0,\,t,\,n-2t-1}=F_{0,\, \frac{n-1}{2},\,0}$).
If $2\leq t\leq \frac{n-2}{2}$ for even $n$, we have
$$n^2+8n+\frac{25}{n}-32\leq BH(F_{0,\,t,\,n-2t-1})\leq \frac{7n^2}{2}-\frac{51n}{4}+\frac{16}{n}+4,$$
the left (right) equality holds if and only if  $F_{0,\,t,\,n-2t-1}=F_{0,\,2,\,n-5}$ ($F_{0,\,t,\,n-2t-1}=F_{0,\, \frac{n-2}{2},\,1}$).

(4) If $s\geq 1$ and $t=0$, by Lemma \ref{le2,10}, we have
$$\phi(L(F_{s,\,0,\,n-2s-1}))=x(x-n)(x-3)^s(x-1)^{n-s-2}.$$
Thus
$$BH(F_{s,\,0,\,n-2s-1})=n^2-\frac{8}{9}sn-2n+\frac{1}{n}.$$
If $1\leq s\leq \frac{n-1}{2}$ for odd $n$, we have
$$\frac{5n^2 }{9}-\frac{14n}{9}+\frac{1}{n}\leq BH(F_{0,\,t,\,n-2t-1})\leq n^2-\frac{26}{9}n+\frac{1}{n},$$
the left (right) equality holds if and only if $F_{s,\,0,\,n-2s-1}=F_{1,\,0,\,n-3}$ ($F_{s,\,0,\,n-2s-1}=F_{\frac{n-1}{2},\,0,\,0}$).
If $1\leq s\leq \frac{n-2}{2}$ for even $n$, we have
$$\frac{5n^2}{9}-\frac{10n}{9}+\frac{1}{n}\leq BH(F_{0,\,t,\,n-2t-1})\leq n^2-\frac{26}{9}n+\frac{1}{n},$$
the left (right) equality holds if and only if $F_{s,\,0,\,n-2s-1}=F_{1,\,0,\,n-3}$ ($F_{s,\,0,\,n-2s-1}=F_{\frac{n-2}{2},\,0,\,1}$).

(5) If $s\geq 1$ and $t\geq 1$, by Lemma \ref{le2,10}, we have
\begin{eqnarray*}
\phi(L(F_{s,\,t,\,n-2s-2t-1})) & = & x(x-3)^{s}(x-1)^{n-s-2t-2}(x^2-3x+1)^{t-1}\\
& & [x^3-(n-t+3)x^2+(3n-3t+1)x-n].
\end{eqnarray*}
By a similar reasoning as the proof of Theorem \ref{th5,1}, we have
$$BH(L(F_{s,\,t,\,n-2s-2t-1}))=n^2-\frac{8}{9}sn-2n+\frac{9t^2+(5n^2-16n-6)t+1}{n}.$$
If $s=1$ and $t=\frac{n-3}{2}$ for odd $n$, we have
$$BH(F_{s,\,t,\,n-2s-2t-1})_{\max}=\frac{7n^2}{2}-\frac{581n}{36}+\frac{121}{4n}+\frac{15}{2},$$
the equality holds if and only if $F_{s,\,t,\,n-2s-2t-1}=F_{1,\,\frac{n-3}{2},\,0}$.\\
If $s=1$ and $t=\frac{n-4}{2}$ for even $n$, we have
$$BH(F_{s,\,t,\,n-2s-2t-1})_{\max}=\frac{7n^2}{2}-\frac{671n}{36}+\frac{49}{n}+11,$$
the equality holds if and only if $F_{s,\,t,\,n-2s-2t-1}=F_{1,\,\frac{n-4}{2},\,1}$.\\
If $s=\frac{n-3}{2}$ and $t=1$ for odd $n$, we have
$$BH(F_{s,\,t,\,n-2s-2t-1})_{\min}=\frac{5n^2}{9}+\frac{13n}{3}+\frac{4}{n}-16,$$
the equality holds if and only if $F_{s,\,t,\,n-2s-2t-1}=F_{\frac{n-3}{2},\,1,\,0}$.\\
If $s=\frac{n-4}{2}$ and $t=1$ for even $n$, we have
$$BH(F_{s,\,t,\,n-2s-2t-1})_{\min}=\frac{5n^2}{9}+\frac{43n}{9}+\frac{4}{n}-16,$$
the equality holds if and only if $F_{s,\,t,\,n-2s-2t-1}=F_{\frac{n-4}{2},\,1,\,1}$.

Combining the above arguments, we have the proof. $\Box$
\end{proof}

\section{\large  The biharmonic index and graph operations }

\begin{lemma}{\bf (\cite{M})}\label{le6,1} 
Let $G$ be a connected graph with $n$ vertices. Then $\lambda_i(\overline{G})=n-\lambda_{n+2-i}(G)$  for $i= 2, \ldots, n$.
\end{lemma}

\begin{theorem}\label{th6,1} 
Let $G$ be a connected graph with $n$ vertices. If $\overline{G}$ is a connected graph, then
$$BH(\overline{G})= n\sum\limits_{i=2}^{n}\frac{1}{(n-\lambda_{n+2-i}(G))^2}.$$
\end{theorem}

\begin{proof}
By Lemma \ref{le6,1}, we have the proof. $\Box$
\end{proof}

The union of two graphs $G_1$ and $G_2$ is the graph $G_1\cup G_2$ with vertex set $V_1(G)\cup V_2(G)$ and edge set $E(G_1)\cup E(G_2)$. The join $G_1\vee G_2$ is obtained from $G_1\cup G_2$ by adding to it all edges between vertices from $V(G_1)$ and $V(G_2)$.

\begin{lemma}{\bf (\cite{M1})}\label{le6,2} 
Let $G_1$ and $G_2$ be graphs on $n_1$ and $n_2$
vertices, respectively. Then the Laplacian eigenvalues of $G_1\vee G_2$ are $n_1+n_2$, $\lambda_i(G_1)+n_2$ ($2\leq i\leq n_1$) and $\lambda_j(G_2)+n_1$ ($2\leq j\leq n_1$).
\end{lemma}

\begin{theorem}\label{th6,2} 
Let $G$ be a connected graph with $n$ vertices. Then
$$BH(G_1\vee G_2)= (n_1+n_2)\left(\frac{1}{(n_1+n_2)^2}+\sum\limits_{i=2}^{n_1}\frac{1}{(\lambda_i(G_1)+n_2)^2}+\sum\limits_{j=2}^{n_2}\frac{1}{(\lambda_j(G_2)+n_1)^2}\right).$$
\end{theorem}

\begin{proof}
By Lemma \ref{le6,2}, we have the proof. $\Box$
\end{proof}

The Cartesian product of $G_1$ and $G_2$ is the graph $G_1\Box G_2$, whose vertex set is $V=V_1\times V_2$ and
where two vertices $(u_i,v_s)$ and $(u_j,v_t)$ are adjacent if and only if either $u_i=u_j$ and $v_sv_t\in E(G_2)$ or $v_s=v_t$ and $u_iu_j\in E(G_1)$.

\begin{lemma}{\bf (\cite{F, M})}\label{le6,3} 
Let $G_1$ and $G_2$ be graphs on $n_1$ and $n_2$
vertices, respectively. Then the Laplacian eigenvalues of $G_1\Box G_2$ are all possible
sums $\lambda_i(G_1)+\lambda_j(G_2)$, $1\leq i \leq n_1$ and $1\leq j \leq n_2$.
\end{lemma}

\begin{theorem}\label{th6,3} 
Let $G_1$ and $G_2$ be two connected graphs. Then
$$BH(G_1\Box G_2)= n_1n_2\left(\sum\limits_{i=2}^{n_1}\frac{1}{\lambda_i^2(G_1)}+\sum\limits_{j=2}^{n_2}\frac{1}{\lambda_j^2(G_2)}
+\sum\limits_{i=2}^{n_1}\sum\limits_{j=2}^{n_2}\frac{1}{(\lambda_i(G_1)+\lambda_j(G_2))^2}\right).$$
\end{theorem}

\begin{proof}
By Lemma \ref{le6,3}, we have the proof. $\Box$
\end{proof}

The lexicographic product $G_1[G_2]$, in which vertices $(u_i,v_s)$ and $(u_j,v_t)$ are adjacent if either $u_iu_j\in E(G_1)$
or $u_i=u_j$ and $v_sv_t\in E(G_2)$ (see \cite{BBP}).

\begin{lemma}{\bf (\cite{BBP})}\label{le6,4} 
Let $G_1$ and $G_2$ be graphs on $n_1$ and $n_2$
vertices, respectively. Then the Laplacian eigenvalues of $G_1[G_2]$ are $n_2\lambda_i(G_1)$ and $\lambda_j(G_2)+d(u_i)n_2$,
where $d(u_i)$ is vertex degree of $G_1$, $1\leq i \leq n_1$ and $2\leq j \leq n_2$.
\end{lemma}

\begin{theorem}\label{th6,4} 
Let $G_1$ and $G_2$ be connected graphs on $n_1$ and $n_2$ vertices, respectively.
$$BH(G_1[G_2])=n_1n_2\left(\sum\limits_{i=2}^{n_1}\frac{1}{n_2^2\lambda_i^2(G_1)}
+\sum\limits_{j=2}^{n_2}\sum\limits_{i=1}^{n_1}\frac{1}{(\lambda_j(G_2)+d_{G_1}(u_i)n_2)^2}\right).$$

\end{theorem}

\begin{proof}
By Lemma \ref{le6,4}, we have the proof. $\Box$
\end{proof}

\section{\large Further work}

The biharmonic eccentricity $\varepsilon_b(u)$ of vertex $u$ in a connected graph $G$ is defined as $\varepsilon_b(u)=\max\{d_{B}^2(u,v)\mid v\in V(G)\}$. Let $d(u)$ be the degree of the corresponding vertex $u$. The following four topological indices will be the problems that need further exploration.\\
(1) The Schultz biharmonic index:
$$SBI(G)=\frac{1}{2}\sum\limits_{u\in V(G)}\sum\limits_{v\in V(G)}(d(u)+d(v))d_{B}^2(u,v).$$
(2) The Gutman biharmonic index:
$$GBI(G)=\frac{1}{2}\sum\limits_{u\in V(G)}\sum\limits_{v\in V(G)}(d(u)d(v))d_{B}^2(u,v).$$
(3) The eccentric biharmonic distance sum:
$$\xi_B(G)=\frac{1}{2}\sum\limits_{u\in V(G)}\sum\limits_{v\in V(G)}(\varepsilon_b(u)+\varepsilon_b(v))d_{B}^2(u,v).$$
(4) The multiplicative eccentricity biharmonic distance:
$$\xi_{B}^{*}(G)=\frac{1}{2}\sum\limits_{u\in V(G)}\sum\limits_{v\in V(G)}(\varepsilon_b(u)\varepsilon_b(v))d_{B}^2(u,v).$$

\small {

}

\end{document}